\documentclass[preprint]{elsarticle}

\makeatletter
\def\ps@pprintTitle{%
	\let\@oddhead\@empty
	\let\@evenhead\@empty
	\def\@oddfoot{\footnotesize\itshape
		{} \hfill\today}%
	\let\@evenfoot\@oddfoot
}
\makeatother

\usepackage{latexsym}
\usepackage{lineno}
\usepackage{hyperref}
\usepackage{indentfirst}
\usepackage{amsxtra}
\usepackage{amssymb}
\usepackage{amsthm}
\usepackage{natbib}
\usepackage{amsmath}
\usepackage{tikz}
\usepackage{amscd}
\usepackage{MnSymbol}
\usepackage[capitalise]{cleveref}
\newtheorem{theor}{Theorem}
\newtheorem{prop}[theor]{Proposition}
\newtheorem{cor}[theor]{Corollary}
\newtheorem{lemma}[theor]{Lemma}

\theoremstyle{definition} 
\newtheorem{defin}{Definition}
\newtheorem{rem}{Remark}

\newtheorem*{conv}{Convention}
\newtheorem{ex}[theor]{Example}

\DeclareMathOperator{\Sym}{Sym}
\DeclareMathOperator{\id}{id}
\DeclareMathOperator{\Aut}{Aut}
\DeclareMathOperator{\Ret}{Ret}

\begin{document}

\begin{frontmatter}
	\title{A characterization of finite simple set-theoretic solutions of the Yang-Baxter equation\tnoteref{mytitlenote}}
	\tnotetext[mytitlenote]{The author is a member of GNSAGA (INdAM).}
	\author[unile]{Marco~CASTELLI}
	\ead{marco.castelli@unisalento.it - marcolmc88@gmail.com}
	%

\begin{abstract}
In this paper we present a characterization of finite simple involutive non-degenerate set-theoretic solutions of the Yang-Baxter equation by means of left braces and we provide some significant examples. 
\end{abstract}
\begin{keyword}
\texttt{set-theoretic solution\sep Yang-Baxter equation\sep brace\sep skew brace}
\MSC[2020] 16T25\sep 81R50
\end{keyword}

\end{frontmatter}

\section*{Introduction}
The quantum Yang-Baxter equation first appeared in theoretical physics, in a paper by C.N. Yang \cite{yang1967}, and in statistical mechanics, in R.J. Baxter's work \cite{bax72}. To date, it is subject of many studies of very current interest in pure mathematics even beyond theoretical physics. In 1992 Drinfel'd \cite{drinfeld1992some} suggested the study of the set-theoretical version of this equation. Specifically, a \emph{set-theoretic solution of the Yang-Baxter equation} on a non-empty set $X$ is a pair $\left(X,r\right)$, where 
$r:X\times X\to X\times X$ is a map such that the relation
\begin{align*}
\left(r\times\id_X\right)
\left(\id_X\times r\right)
\left(r\times\id_X\right)
= 
\left(\id_X\times r\right)
\left(r\times\id_X\right)
\left(\id_X\times r\right)
\end{align*}
is satisfied.  
Writing a solution $(X,r)$ as $r\left(x,y\right) = \left(\lambda_x\left(y\right)\rho_y\left(x\right)\right)$, with
$\lambda_x, \rho_x$ maps from $X$ into itself, for every $x\in X$, we say that $(X, r)$ is \emph{non-degenerate} if $\lambda_x,\rho_x\in \Sym_X$, for every $x\in X$, and \emph{involutive} if $r^2=\id_{X\times X}$.\\
The papers of Gateva-Ivanova and Van Den Bergh  \cite{gateva1998semigroups} and Etingov, Schedler, and Soloviev \cite{etingof1998set} attracted several authors to the study of the involutive non-degenerate set-theoretic solutions (which we simply call \emph{solutions}). In particular, in \cite[Section 2]{etingof1998set} the class of \textit{indecomposable solutions} was introduced. The interest in these solutions is motivated by the fact that they allow to construct other solutions, not necessarily indecomposable, by suitable construction-tools, such as dynamical extension and retraction-process (see \cite{cacsp2018,etingof1998set,vendramin2016extensions} for more details). 
In this context, many theory has been developed and several interesting results have been obtained (see, for example, \cite{cacsp2018,castelli2021indecomposable,capiru2020,cedo2020primitive,JePiZa20x,jedlivcka2021cocyclic,rump2020,Ru20,smock} and related references).\\
Among indecomposable solutions, the \emph{simple} ones play a special role since they are the "fundamental blocks" to construct all the others by dynamical extensions (see \cite[Proposition 2]{cacsp2018} and \cite[Corollary 2.13]{vendramin2016extensions}). On the other hand, until a year ago, very little was known about simple solutions. In particular, there were no methods for building simple solutions and only a few number of simple solutions was known: the simple solutions having size $\leq 8$, obtained by inspection of the database of small solutions, and the indecomposable ones having prime size $p$ (where $\lambda_x=\rho^{-1}_x$ and all the $\lambda_x$ are all equal to the same $p$-cycle). In the last year, Ced\'o and Okni\'nski provided a construction-methods that allow to obtain other families of simple solutions \cite{cedo2021constructing,cedo2022new} for the first time. In particular, they exhibited simple solutions having size $n^2$ and $m n^2$, for $n,m>1$. For these constructions, the imprimitivity action of the permutation group associated to these solutions has a crucial importance (see \cite{cedo2020primitive} for more details). Following a different approach, in \cite{cedo2022new} they found a striking link between simple left braces and simple solutions, showing that a large family of simple left braces provides simple solutions \cite[Theorem 5.1]{cedo2022new}. Actually, a natural problem, recently posed by Okni\'nski during a talk in \cite{Okninski2022}, is the construction of different simple solutions.\\
The aim of this paper is to give a characterization of finite simple solutions by means of left braces, algebraic structures introduced by Rump in \cite{rump2007braces}. This naturally provides the precise class of left braces that give rise to simple solutions. The paper is organized as follows. In Section $1$ we introduce the necessary background on left braces and on solutions, using the language of cycle sets. In Section $2$ we show the main result, which give a characterization of simple solutions by means of left braces. In Section $3$ we apply our result to provide some examples of simple solutions, some of which are different from the ones obtained in \cite{cedo2021constructing,cedo2022new}.

\section{Basic results}

In this section, we mainly recall some basics on solutions, cycle sets and left braces that are useful throughout the paper. 

\subsection{Solutions of the Yang-Baxter equation and cycle sets}
In \cite{rump2005decomposition}, Rump found a one-to-one correspondence between solutions and an algebraic structure with a single binary operation, which he called \emph{non-degenerate cycle sets}.
To illustrate this correspondence, let us firstly recall the following definition.

\begin{defin}[pag. 45, \cite{rump2005decomposition}]
A pair $(X,\cdot)$ is said to be a \emph{cycle set} if each left multiplication $\sigma_x:X\longrightarrow X,$ $y\mapsto x\cdot y$ is bijective and 
$$(x\cdot y)\cdot (x\cdot z)=(y\cdot x)\cdot (y\cdot z), $$
for all $x,y,z\in X$.  Moreover, a cycle set $(X,\cdot)$ is called \textit{non-degenerate} if the squaring map $x\mapsto x\cdot x$ is bijective.
\end{defin}

\begin{conv}
Even if not specified, all the cycle sets are finite and non-degenerate throughout the paper.
\end{conv}

\begin{prop}[Propositions 1-2, \cite{rump2005decomposition}]\label{corrisp}
Let $(X,\cdot)$ be a cycle set. Then the pair $(X,r)$, where $r(x,y):=(\sigma_x^{-1}(y),\sigma_x^{-1}(y)\cdot x)$, for all $x,y\in X$, is a solution of the Yang-Baxter equation which we call the associated solution to $(X,\cdot)$. Moreover, this correspondence is one-to-one.
\end{prop}

A first useful tool to construct new solutions, introduced in \cite{etingof1998set}, is the so-called \textit{retract relation}, an equivalence relation on $X$ which we denote by $\sim_r$. Precisely, if $(X,r)$ is a solution, then $x\sim_r y$ if and only if $\lambda_x=\lambda_y$, for all $x,y\in X$. In this way, $(X,r)$ induces another solution, having the quotient $X/\sim_r$ as underlying set, which is named \textit{retraction} of $(X,r)$ and is denoted by $\Ret(X, r)$. 
As one can expect, the retraction of a solution corresponds to the retraction of a non-degenerate cycle set. Specifically, in \cite{rump2005decomposition} Rump showed that the binary relation $\sim_\sigma$ on $X$ given by 
$$x\sim_\sigma y :\Longleftrightarrow \sigma_x = \sigma_y$$ 
for all $x,y\in X$, is a \emph{congruence} of $(X,\cdot)$, i.e. an equivalence relation for which $x\sim_\sigma y$ and $x'\sim_\sigma y'$ implies $x\cdotp x'\sim_\sigma y\cdotp y',$ for all $x,x',y,y'\in X$. Moreover, he proved that the quotient $X/\sim_{\sigma}$, which we denote by $\Ret(X)$, is a cycle set and he called it the \emph{retraction} of $(X,\cdot)$. As the name suggests, if $(X, \cdot)$ is the cycle set associated to a solution $(X,r)$, then the retraction $\Ret(X)$ is the cycle set associated to $\Ret(X,r)$. Besides, a cycle set $X$ is said to be \textit{irretractable} if $\Ret(X)=X$, otherwise it is called \textit{retractable}.

\smallskip



\noindent For a cycle set $X$, the permutation group generated by the set $\{\sigma_x \, | \, x\in X\}$ will be denoted by $\mathcal{G}(X)$ and we call it the \textit{associated permutation group}. Obviously, in terms of solutions, the associated permutation group is exactly the permutation group generated by the set $\{\lambda_x \, | \, x\in X\}$.
\medskip

\noindent Our attention is mainly posed on cycle sets that are indecomposable. 
\begin{defin}
A cycle set $(X,\cdot)$ is said to be \textit{indecomposable} if the permutation group $\mathcal{G}(X)$ acts transitively on $X$. 
\end{defin}

\begin{rem}
Note that a solution $(X,r)$ is indecomposable if and only if the associated cycle set $(X,\cdot)$ is indecomposable. Now, in the rest of the paper, we will study indecomposable solutions using the language of cycle sets. However, all the results involving cycle sets can be translated in terms of solutions by \cref{corrisp}.
\end{rem} 

\medskip

\noindent In a classical way one can define the notion of homomorphism between two cycle sets $X,Y$, i.e. a function $f$ from $X$ to $Y$ such that $f(x\cdotp y)=f(x)\cdotp f(y)$ for all $x,y\in X$. A surjective homomorphism is said to be \emph{epimorphism}, while a bijective homomorphism is said to be \emph{isomorphism}. It is easy to show that if $h$ is an epimorphism from a cycle set $Y$ to a cycle set $Z$, then the binary relation given by $x\sim_h y:\Longleftrightarrow h(x)=h(y)$ is a congruence of $(Y,\cdotp)$ and conversely every congruence of $Y$ give rise to an epimorphism of cycle sets. Now, we can give the key-notion of this paper.

\begin{defin}
A cycle set $X$ is said to be \emph{simple} if $|X|>1$ and for every epimorphism $f$ from $X$ to a cycle set $Y$ we have that $|Y|=1$ or $f$ is an isomorphism.
\end{defin}
We remark that even if this definition of simplicity, given in \cite{cedo2021constructing}, is different from the original one given by Vendramin in \cite{vendramin2016extensions}, by \cite[Lemma 1]{cacsp2018} both definitions coincide for finite indecomposable cycle sets. However, we have the following result.

\begin{prop}(\cite[Lemma 4.1 - Proposition 4.2]{cedo2021constructing})\label{indirr}
Let $X$ be a simple cycle set. Then, if $|X|>2$, $X$ is indecomposable. Moreover, if $|X|$ is not a prime number, $X$ is irretractable.
\end{prop}

\subsection{Solutions of the Yang-Baxter equation and left braces}
At first, we introduce the following definition that, as observed in \cite{cedo2014braces}, is equivalent to the original introduced by Rump in \cite{rump2007braces}.

\begin{defin}[\cite{cedo2014braces}, Definition 1]
A set $A$ endowed of two operations $+$ and $\circ$ is said to be a \textit{left brace} if $(A,+)$ is an abelian group, $(A,\circ)$ a group, and
$$
    a\circ (b + c) + a
    = a\circ b + a\circ c,
$$
for all $a,b,c\in A$.
\end{defin}

\begin{ex}\label{esbrace}
If $X$ is a cycle set, then one can show that the permutation group $\mathcal{G}(X)$ give rise to a left brace $(\mathcal{G}(X),+,\circ)$, where $\circ$ is the usual composition in $\mathcal{G}(X)$ (see, for example, \cite[Section 2]{bachiller2015family} for more details). From now on, we will refer to $(\mathcal{G}(X),+,\circ)$ as the \emph{associated left brace}.
\end{ex}

Given a left brace $A$ and $a\in A$, let us denote by $\lambda_a:A\longrightarrow A$ the map from $A$ into itself defined by
\begin{equation*}\label{eq:gamma}
    \lambda_a(b):= - a + a\circ b,
\end{equation*} 
for all $b\in A$. 
As shown in \cite[Proposition 2]{rump2007braces} and \cite[Lemma 1]{cedo2014braces}, these maps have special properties. We recall them in the following proposition.
\begin{prop}\label{action}
Let $A$ be a left brace. Then, the following are satisfied: 
\begin{itemize}
\item[1)] $\lambda_a\in\Aut(A,+)$, for every $a\in A$;
\item[2)] the map $\lambda:A\longrightarrow \Aut(A,+)$, $a\mapsto \lambda_a$ is a group homomorphism from $(A,\circ)$ into $\Aut(A,+)$.
\end{itemize}
\end{prop}

\noindent The map $\lambda$ is of crucial importance to construct cycle sets (and hence solutions of the Yang-Baxter equation) using left braces, as one can see in the following proposition.

\begin{prop}[Lemma 2, \cite{cedo2014braces}]
Let A be a left brace and $\cdotp$ the binary operation on $A$ map given by
$$
a\cdotp b:=\lambda_a^{-1}(b), $$
for all $a,b\in A$. Then, $(A,\cdotp)$ is a cycle set.
\end{prop}

\noindent Therefore, the previous proposition and \cref{esbrace} show that if $X$ is a cycle set, then we can construct a cycle set on the permutation group $\mathcal{G}(X)$.\\
For the following definition, we refer the reader to \cite[pg. 160]{rump2007braces} and \cite[Definition 3]{cedo2014braces}.

\begin{defin}
Let $A$ be a left brace. A subset $I$ of $A$ is said to be a \textit{left ideal} if it is a subgroup of the multiplicative group and $\lambda_a(I)\subseteq I$, for every $a\in A$. Moreover, a left ideal is an \textit{ideal} if it is a normal subgroup of the multiplicative group.
\end{defin}
%
\noindent As one can expect, if $I$ is an ideal of a left brace $A$, then the structure $A/I$ is a left brace called the \emph{quotient left brace} of $A$ modulo $I$. Moreover, the ideal $\{0\}$ will be called the \emph{trivial} ideal and and a left brace $A$ which contains no ideals different from $\{0\}$ and $A$ will be called a \emph{simple} left brace.\\
In \cite{rump2007braces}, Rump introduced the special notion of the socle of a left brace that, in the terms of \cite[Section 4]{cedo2014braces}, is the following.
\begin{defin}
Let $A$ be a left brace. Then, the set 
$$
    Soc(A) := \{a\in A \ | \ \forall \,
     b\in A \quad  a + b = a\circ b \}
$$
is named \emph{socle} of $A$.
\end{defin}
\noindent Clearly,
$Soc(A) := \{a\in A \ | \ \lambda_a = \id_A\}$.
Moreover, we have that $Soc(A)$ is an ideal of $A$. If a cycle set $X$ is irretractable, then the socle of its associated left brace is trivial.

\begin{prop}(Lemma 2.1, \cite{bachiller2015family})\label{soczero}
Let $X$ be an irretractable cycle set. Then, $Soc(\mathcal{G}(X))=\{id_X\}$. 
\end{prop}

\noindent The socle is useful to establish a link between a left brace $A$ and the permutation group associated to the cycle set $(A,\cdotp)$.

\begin{prop}(Lemma 2.2, \cite{cedo2021constructing})\label{isosocle}
Let $A$ be a left brace and $\mathcal{G}(A)$ the permutation group associated to the cycle set $(A,\cdotp)$. Then, as left braces, $A/Soc(A)\cong \mathcal{G}(A)$. 
\end{prop}

If $A$ is a left brace, by \cref{action}, the maps $\lambda_a$ determine an action of $(A,\circ)$ on $(A,+)$. According to 
\cite{rump2007braces, rump2020}, a subset $X$ of $A$ which is a union of orbits with respect to such an action
and generating the additive group $(A,+)$ is called \textit{cycle base}. Moreover, if a cycle base is a single orbit then is said to be a \textit{transitive cycle base}. The following result, that is useful for our scopes, is implicitly contained in \cite{cedo2021constructing}.

\begin{prop}[Section 4, \cite{cedo2021constructing}]\label{injtrans}
Let $X$ be an indecomposable and irretractable cycle set and $\mathcal{G}(X)$ be the left brace on the associated permutation group. Then, $X$ can be regarded as a transitive cycle base of $\mathcal{G}(X)$.
\end{prop}

\noindent We close this preliminary section showing the Galois-correspondence between ideals of a left brace $A$ and congruences of a transitive cycle base $X$ founded by Rump in \cite{rump2006modules,rump2007braces}. This is essentially a mixture between \cite[Theorem 1]{rump2007braces} and \cite[Theorem 1]{rump2006modules}, adapted for our aim (we remark that, in this context, all the cycle sets are finite). At first, recall that if $\sim$ and $\sim'$ are equivalence relations on a set $S$, then $\sim$ is said to be a \emph{refinement} of $\sim'$ if $x\sim y$ implies $x\sim' y$, for all $x,y\in S$.

\begin{theor}[Theorem 1, \cite{rump2007braces} - Theorem 1, \cite{rump2006modules}]\label{prelicong}
Let $A$ be a finite left brace such that $Soc(A)=\{0\}$ and $X$ be a transtive cycle base. Then, if $I$ is an ideal of $A$, the orbits of $X$ under the action of $I$ give rise to a congruence of $X$. Conversely, if $\sim$ is a congruence of $X$, the additive subgroup $I'$ of $A$ generated by $\{x-y\ |\ x,y\in X,\ x\sim y \}$ is an ideal and the orbits of $X$ under the action of $I'$  give rise to a congruence of $X$ which is a refinement of $\sim$.
\end{theor}

\section{Simple solutions and left braces}

This section is devoted to provide a characterization of finite simple solutions. We start with a couple of lemma.

\begin{lemma}\label{prelcar}
Let $B$ be a finite left brace, with $|B|>1$, such that $Soc(B)=\{0\}$ and having a transitive cycle base $X$. Then, the left brace $\mathcal{G}(X)$ is isomorphic to $B$.
\end{lemma}

\begin{proof}
Clearly, the permutation $\sigma_x$ in $\mathcal{G}(X)$ is equal to the element $\lambda_x^{-1}$ in $\mathcal{G}(B)$ restricted to the set $X$.
Since the natural embedding from $X$ to $B$ induces the left braces homomorphism $f$ from $\mathcal{G}(X)$ to $\mathcal{G}(B)$ given by $f(\sigma_x^{-1}):=\lambda_x$ for all $x\in X$ (see \cite[Section 2]{cedo2022new} for more details), by \cref{isosocle} it is sufficient to show that $f$ is bijective. As done in the first part of the proof of \cite[Theorem 5.1]{cedo2022new} (where $f$ is denoted by $\bar{i}$), one can show that $f$ is surjective. Now, if $g\in \mathcal{G}(X)$ is such that $g\in Ker(f)$, clearly $g(x)=x$ for all $x\in X$, hence $f$ is injective and the thesis follows.
\end{proof}

\begin{lemma}(Proposition 4.3, \cite{cedo2021constructing})\label{preid}
Let $X$ be a simple cycle set such that $|X|$ has not prime size. Then, the left brace $\mathcal{G}(X)$ has a unique minimal ideal $I$ equal to the additive subgroup generated by $\{\sigma_x-\sigma_y\ | \ x,y\in X \}$. Moreover, $\mathcal{G}(X)/I$ is a trivial cyclic left brace.
\end{lemma}



\noindent Now, we are able to show the main result of the section.

\begin{theor}\label{bracestransi}
Let $B$ be a finite left brace, with $|B|>1$, such that $Soc(B)=\{0\}$ and having a transitive cycle base $X$. Then, the following conditions are equivalent:
\begin{itemize}
    \item[1)] $X$ is a simple cycle set;
    \item[2)] every non-trivial ideal of $B$ acts transitively on $X$;
    \item[3)] $B$ has a unique minimal ideal $I$ that acts transitively on $X$.
\end{itemize}
\end{theor}

\begin{proof}
If $X$ is simple and $B$ has a non-trivial ideal $I$ that does not acts transitively on $X$, then by \cref{prelicong} induces a congruence $\sim_I$ on $X$. Therefore, $x\sim_I y$ if and only if $x=y$ and, since $X$ is a transitive cycle base, it follows that $I\subseteq Soc(B)$, a contradiction. Hence we showed that $1)$ implies $2)$.\\
Conversely, suppose that $X$ is not simple. By hypothesis, one can easily show that $|X|>1$.  Therefore, there exist a non-trivial congruence $\sim_P$ (here, non-trivial means that $\sim_P$ does not induces a trivial partition of $X$). By \cref{prelicong}, the additive subgroup $I$ of $B$ generated by $\{x-y\in X\ \mid \ x\sim_P y \}$ is an ideal such that $\sim_I$ is a refinement of $\sim_P$. Moreover, since $\sim_P$ is not a trivial congruence, $I$ can not be the trivial ideal. Since $I$ acts transitively on $X$, $x\sim_I y$ for all $x,y\in X$ and hence $x \sim_P y$ for all $x,y\in X$, a contradiction, therefore $2)$ implies $1)$.\\ 
It remains to show that $1)$ is equivalent to $3)$. Suppose $X$ simple. By \cref{prelcar} we have that $B\cong \mathcal{G}(X)$, hence $|X|$ can not be a prime number and by \cref{preid} $B$ has a unique minimal ideal $I$. Since we showed that $1)$ is equivalent to $2)$, it follows that $I$ acts transitively on $X$. Finally, if $B$ has a unique minimal ideal $I$ that acts transitively on $X$, it follows that every non-trivial ideal $J$ of $B$ acts transitively on $X$, hence $3)\Rightarrow 2)$, which implies $3)\Rightarrow 1)$. 
\end{proof}

\noindent By the previous result, we have that \cite[Theorem 5.1]{cedo2022new} follows as a corollary.

\begin{cor}\label{corcedo}\cite[Theorem 5.1]{cedo2022new}
Let $B$ be a simple non-trivial left brace having a transitive cycle base $X$. Then, $X$ is a simple cycle set.
\end{cor}

\begin{proof}
Since $B$ is simple, the unique ideal different from $\{0\}$ is $B$, that acts transitively on $X$. Then, the thesis follow by \cref{bracestransi}.
\end{proof}

\noindent Before giving the announced characterization, which closes the section, we recall that in \cite[Theorem 2.13]{etingof1998set} Etingof, Schedler and Soloviev proved that there is, up to isomorphisms, a unique indecomposable cycle set of size $p$ for every prime number $p$.

\begin{cor}
Let $X$ be a cycle set. If $X$ has size $2$, then $X$ is simple. If $|X|>2$, the following conditions are equivalent:
\begin{itemize}
    \item[1)] $X$ is a simple cycle set;
    \item[2)] $X$ is the unique indecomposable cycle set of size $p$ (for some prime number $p$) or $X$ is an irretractable and indecomposable cycle set such that every non-trivial ideal of $\mathcal{G}(X)$ acts transitively on $X$;
    \item[3)] $X$ is the unique indecomposable cycle set of size $p$ (for some prime number $p$) or $X$ is an irretractable and indecomposable cycle set such that $\mathcal{G}(X)$ has a unique minimal ideal $I$ that acts transitively on $X$.
\end{itemize}
\end{cor}


\begin{proof}
The first part trivially follows. Suppose that $X$ is a simple cycle set with $|X|>2$. Then, by \cref{indirr} it is indecomposable. If $X$ has not prime size, again by \cref{indirr} it also is irretractable. Therefore, by \cref{soczero}, \cref{injtrans} and \cref{bracestransi} (applied with $B:=\mathcal{G}(X)$) we obtain $1)\Rightarrow 2)$ and $1)\Rightarrow 3)$. \\
If $X$ is an indecomposable cycle set of size $p$, for some prime number $p$, then it is simple by \cite[Lemma 1]{cacsp2018}. While, if $X$ is an irretractable and indecomposable cycle set such that every non-trivial ideal of $\mathcal{G}(X)$ acts transitively on $X$, again by \cref{soczero}, \cref{injtrans} and \cref{bracestransi} (applied with $B:=\mathcal{G}(X)$) we obtain $2)\Rightarrow 1)$. In a similar way, one can show $3)\Rightarrow 1)$.
\end{proof}

\section{Examples}

In this section, we apply \cref{bracestransi} to some left braces (found using a GAP package  \cite{Ve15pack}) to provide concrete examples of simple cycle sets, some of which are different from those given in \cite{cedo2021constructing,cedo2022new}. 

\smallskip 

In order to start with a small example, we show providing a well-known cycle set of size $4$.

\begin{ex}\label{ex1}
Let $B$ be the left brace $B_{8,27}$ of \cite{Ve15pack}. Then, it is a left brace of size $8$ having only one non-trivial ideal $I$, which has size $4$, different from $B$. Moreover, it has a transitive cycle base $X$ having size $4$. In particular, $X$ is isomorphic to the cycle set given by $C:=\{1,2,3,4\}$ and 
$$
\sigma_1:=( 2,4)
$$
$$
\sigma_2:=( 1,3)
$$
$$
\sigma_3:= ( 1, 2,3,4)
$$
$$
\sigma_4:=( 1,4,3,2).
$$
Since $I$ acts transitively on $X$, the cycle set $C$ is simple.
\end{ex}

In \cite{cedo2022new} a families of simple cycle sets having size $4p_{1}^2$ or $p_1^{2}\cdotp \dots \cdotp p_n^{2}$ , where $p_1,\dots ,p_n$ are prime numbers, and with a simple associated left brace were exhibited. In the following, we give a simple cycle set having a simple associated left brace which is smaller than the ones provided in \cite{cedo2022new}.

\begin{ex}
Let $B$ be the left brace $B_{24,94}$ of \cite{Ve15pack}. Then, it is a simple left brace of size $24$. It has a transitive cycle base $X$ having size $12$. In particular, $X$ is isomorphic to the cycle set given by $C:=\{1,...,12\}$ and 

$$
\sigma_1:=( 1, 8)( 2,10)( 3,11)( 4, 9)( 5, 6)( 7,12)
$$
$$
\sigma_2:= ( 1, 9)( 2, 5, 6,10)( 3,11,12, 7)( 4, 8)
$$
$$
\sigma_3:=  ( 1, 9)( 2,10, 6, 5)( 3, 7,12,11)( 4, 8)
$$
$$
\sigma_4:=( 1, 2,12)( 3, 4, 6)( 5, 9,11)( 7,10, 8)
$$
$$
\sigma_5:=  ( 1,10, 4, 5)( 2,11)( 3, 8,12, 9)( 6, 7)
$$
$$\sigma_6:=( 1, 6,12)( 2, 3, 4)( 5, 9, 7)( 8,11,10)
$$
$$
\sigma_7:=   ( 1,10)( 2, 7)( 3, 9)( 4, 5)( 6,11)( 8,12)
$$
$$
\sigma_8:=( 1, 5, 4,10)( 2,11)( 3, 9,12, 8)( 6, 7)
$$
$$
\sigma_9:= ( 1, 3, 6)( 2, 4,12)( 5,11, 8)( 7, 9,10)
$$
$$
\sigma_{10}:= ( 1,12, 6)( 2, 4, 3)( 5, 7, 9)( 8,10,11)
$$
$$
\sigma_{11}:=   ( 1,12, 2)( 3, 6, 4)( 5,11, 9)( 7, 8,10)
$$
$$
\sigma_{12}:=( 1, 6, 3)( 2,12, 4)( 5, 8,11)( 7,10, 9)
$$
By \cref{corcedo}, $X$ is simple. Moreover, by \cref{prelcar}, $\mathcal{G}(X)$ is a simple left brace.
\end{ex}

\begin{ex}\label{ex2}
Let $B$ be the left brace $B_{48,1532}$ of \cite{Ve15pack}. Then, it is a left brace of size $48$ having only one non-trivial ideal $I$, which has size $24$, different from $B$. Moreover, it has a transitive cycle base $X$ having size $12$. In particular, $X$ is isomorphic to the cycle set given by $C:=\{1,...,12\}$ and 

$$
\sigma_1:=( 1, 8,11, 6)( 2, 5)( 3, 4, 9,10)( 7,12)
$$
$$
\sigma_2:=( 1, 4)( 2, 9,12, 3)( 5, 8, 7, 6)(10,11)
$$
$$
\sigma_3:= ( 1, 6,11, 8)( 2, 5)( 3,10, 9, 4)( 7,12)
$$
$$
\sigma_4:=( 1,10)( 2, 3)( 4,11)( 5, 6)( 7, 8)( 9,12)
$$
$$
\sigma_5:=  ( 1, 6)( 2, 7)( 3, 4)( 5,12)( 8,11)( 9,10)
$$
$$\sigma_6:=( 1, 4)( 2, 3,12, 9)( 5, 6, 7, 8)(10,11)
$$
$$
\sigma_7:=  ( 1, 3, 5,11, 9, 7)( 2,10, 8,12, 4, 6)
$$
$$
\sigma_8:=( 1, 5, 3,11, 7, 9)( 2, 6,10,12, 8, 4)
$$
$$
\sigma_9:= ( 1, 3, 7,11, 9, 5)( 2, 4, 6,12,10, 8)
$$
$$
\sigma_{10}:= ( 1, 7, 9,11, 5, 3)( 2, 6, 4,12, 8,10)
$$
$$
\sigma_{11}:=   ( 1, 9, 7,11, 3, 5)( 2, 4, 8,12,10, 6)
$$
$$
\sigma_{12}:= ( 1, 5, 9,11, 7, 3)( 2, 8,10,12, 6, 4)
$$
Since $I$ acts transitively on $X$, the cycle set $C$ is simple.
\end{ex}

\begin{ex}\label{ex3}
Let $B$ be the left brace $B_{32,24526}$ of \cite{Ve15pack}. Then, it is a left brace of size $32$ having only one non-trivial ideal $I$, which has size $16$, different from $B$. Moreover, it has a transitive cycle base $X$ having size $16$. In particular, $X$ is isomorphic to the cycle set given by $C:=\{1,...,16\}$ and 

$$
\sigma_1:=( 1,13)( 2, 4)( 3,15)( 6,12)( 8,10)(14,16)
$$
$$
\sigma_2:=( 1, 3)( 2,14)( 4,16)( 5,11)( 7, 9)(13,15)
$$
$$
\sigma_3:=( 1,11,13, 7)( 2, 6,14,10)( 3, 9,15, 5)( 4, 8,16,12)
$$
$$
\sigma_4:=( 1, 5,13, 9)( 2,12,14, 8)( 3, 7,15,11)( 4,10,16, 6)
$$
$$
\sigma_5:=   ( 1, 7,13,11)( 2,10,14, 6)( 3, 5,15, 9)( 4,12,16, 8)
$$
$$\sigma_6:= ( 1, 9,13, 5)( 2, 8,14,12)( 3,11,15, 7)( 4, 6,16,10)
$$
$$
\sigma_7:=  ( 2,16)( 4,14)( 5, 9)( 6, 8)( 7,11)(10,12)
$$
$$
\sigma_8:=( 1,15)( 3,13)( 5, 7)( 6,10)( 8,12)( 9,11)
$$
$$
\sigma_9:= ( 1, 8, 3, 6)( 2, 9, 4,11)( 5,16, 7,14)(10,13,12,15)
$$
$$
\sigma_{10}:= ( 1,10, 3,12)( 2, 7, 4, 5)( 6,15, 8,13)( 9,14,11,16)
$$
$$
\sigma_{11}:=   ( 1, 2,15,16)( 3, 4,13,14)( 5,10,11, 8)( 6, 7,12, 9)
$$
$$
\sigma_{12}:= ( 1,16,15, 2)( 3,14,13, 4)( 5, 8,11,10)( 6, 9,12, 7)
$$
$$
\sigma_{13}:= ( 1,14,15, 4)( 2, 3,16,13)( 5, 6,11,12)( 7, 8, 9,10)
$$
$$
\sigma_{14}:=  ( 1, 4,15,14)( 2,13,16, 3)( 5,12,11, 6)( 7,10, 9, 8)
$$
$$
\sigma_{15}:= ( 1,12, 3,10)( 2, 5, 4, 7)( 6,13, 8,15)( 9,16,11,14)
$$
$$
\sigma_{16}:=  ( 1, 6, 3, 8)( 2,11, 4, 9)( 5,14, 7,16)(10,15,12,13)
$$
Since $I$ acts transitively on $X$, the cycle set $C$ is simple.
\end{ex}

\begin{ex}\label{ex4}
Let $B$ be the left brace $B_{81,705}$ of \cite{Ve15pack}. Then, it is a left brace of size $81$ having only one non-trivial ideal $I$, which has size $27$, different from $B$. Moreover, it has a transitive cycle base $X$ having size $27$. In particular, $X$ is isomorphic to the cycle set given by $C:=\{1,...,27\}$ and 
$$
\sigma_{1}:=( 1,17,18)( 2,20,13)( 3,23, 5)( 4,27,16)( 6,14,25)( 7,22,15)( 8,19,24)( 9,12,10)(11,26,21)
$$
$$
 \sigma_{2}:= ( 1, 7,10)( 2,23, 8)( 3,15,11)( 4,24,12)( 5, 6, 9)(13,18,14)(16,22,20)(17,19,21)(25,26,27)
$$
$$
 \sigma_{3}:= ( 1, 8,26)( 2,27, 7)( 3,24,13)( 4,14,11)( 5,22,21)( 6,20,17)( 9,16,19)(10,23,25)(12,18,15)
$$
$$
 \sigma_{4}:= ( 1,26,11,23,10,12,27, 2,13)( 3,17,15, 4, 5,24,18,16,14)( 6, 7,20,22, 8,21,19,25, 9)
$$
$$
 \sigma_{5}:= ( 1,18, 8,23, 3,25,27, 4, 7)( 2, 5,20,26,16,21,10,17, 9)( 6,13,24,22,11,14,19,12,15)
$$
$$
  \sigma_{6}:=( 1, 9,24)( 2, 3,19)( 4, 6,26)( 7, 8,25)(10,18,22)(11,13,12)(14,23,20)(15,27,21)
$$
$$
 \sigma_{7}:= ( 1,11, 6)( 2,15,17)( 3,20,25)( 4,21, 7)( 5,26,24)( 8,18, 9)(10,14,16)(12,22,23)(13,19,27)
$$
$$
 \sigma_{8}:= ( 1, 3, 7,23, 4, 8,27,18,25)( 2,16, 9,26,17,20,10, 5,21)( 6,11,15,22,12,24,19,13,14)
 $$
 $$
 \sigma_{9}:= ( 1,15, 9)( 2,10,26)( 5,13,25)( 6,22,19)( 7,16,11)( 8,17,12)(14,21,27)(20,23,24)
  $$
  $$
 \sigma_{10}:= ( 1,20,15)( 2, 4,22)( 3, 6,10)( 5,16,17)( 7,25, 8)( 9,14,27)(18,19,26)(21,24,23)
  $$
  $$
 \sigma_{11}:= ( 1,12,19)( 2,24,16)( 3,21, 8)( 4, 9,25)( 5,10,15)( 6,23,13)( 7,18,20)(11,22,27)(14,17,26)
  $$
  $$
 \sigma_{12}:= ( 1, 4,25,23,18, 7,27, 3, 8)( 2,17,21,26, 5, 9,10,16,20)( 6,12,14,22,13,15,19,11,24)
  $$
  $$
\sigma_{13}:=  ( 1,23,27)( 2,22,18)( 3,26,19)( 4,10, 6)( 5,25,12)( 7,13,16)( 8,11,17)(14,24,15)
  $$
  $$
\sigma_{14}:=  ( 1,19, 5,23, 6,16,27,22,17)( 2, 8,15,26,25,24,10, 7,14)( 3,13,21, 4,11, 9,18,12,20)
  $$
  $$
\sigma_{15}:=  ( 1,14,20)( 2,26,10)( 3,18, 4)( 5,12, 7)( 8,16,13)( 9,27,24)(11,25,17)(15,21,23)
  $$
  $$
\sigma_{16}:=  ( 1,27,23)( 2,19, 4)( 3,10,22)( 5, 7,11)( 6,18,26)( 8,12,16)( 9,20,21)(13,17,25)
  $$
  $$
\sigma_{17}:=  ( 1,13,22)( 2,14, 5)( 3, 9, 7)( 4,20, 8)( 6,27,12)(10,24,17)(11,19,23)(15,16,26)(18,21,25)
  $$
  $$
\sigma_{18}:=  ( 1,21,14)( 2,18, 6)( 3,22,26)( 4,19,10)( 5,17,16)( 9,15,23)(11,12,13)(20,24,27)
  $$
  $$
\sigma_{19}:=  ( 1,10,13,23, 2,11,27,26,12)( 3, 5,14, 4,16,15,18,17,24)( 6, 8, 9,22,25,20,19, 7,21)
  $$
  $$
 \sigma_{20}:= ( 1,16, 3)( 2, 9,11)( 4,23,17)( 5,18,27)( 6,24, 7)( 8,22,14)(10,21,13)(12,26,20)(15,25,19)
  $$
  $$
 \sigma_{21}:= ( 1, 6,17,23,22, 5,27,19,16)( 2,25,14,26, 7,15,10, 8,24)( 3,11,20, 4,12,21,18,13, 9)
  $$
  $$
\sigma_{22}:=  ( 1,25, 2)( 3,14,12)( 4,15,13)( 5,19,20)( 6,21,16)( 7,26,23)( 8,10,27)( 9,17,22)(11,18,24)
  $$
  $$
 \sigma_{23}:= ( 1,22,16,23,19,17,27, 6, 5)( 2, 7,24,26, 8,14,10,25,15)( 3,12, 9, 4,13,20,18,11,21)
  $$
  $$
 \sigma_{24}:= ( 1, 5, 4)( 2,21,12)( 3,27,17)( 6,15, 8)( 7,19,14)( 9,13,26)(10,20,11)(16,18,23)(22,24,25)
  $$
  $$
\sigma_{25}:=  ( 2, 6, 3)( 4,26,22)( 5, 8,13)( 7,12,17)( 9,21,20)(10,19,18)(11,16,25)(14,15,24)
  $$
  $$
\sigma_{26}:=  ( 1, 2,12,23,26,13,27,10,11)( 3,16,24, 4,17,14,18, 5,15)( 6,25,21,22, 7, 9,19, 8,20)
  $$
  $$
\sigma_{27}:=  ( 1,24,21)( 3, 4,18)( 5,11, 8)( 6,19,22)( 7,17,13)( 9,23,14)(12,25,16)(15,20,27)
$$
Since $I$ acts transitively on $X$, the cycle set $C$ is simple.
\end{ex}

\begin{rem}
We conclude the paper remarking that in \cite{cedo2021constructing,cedo2022new} three types of simple cycle sets were provided: cycle sets that are transitive cycle bases of simple left braces, cycle sets of size $n^2$ and cycle sets having size $n^2 m$ and an element $x$ such that $\sigma_x$ is a $n^2 m$-cycle. Therefore, \cref{ex2} and \cref{ex4} are simple cycle sets that are different from the ones obtained in \cite{cedo2021constructing,cedo2022new}.
\end{rem}

\section*{Acknowledgment}

\noindent The author thank Leandro Vendramin for the informations about his GAP Package and Paola Stefanelli for the discussion about left braces.

\bibliographystyle{elsart-num-sort}
\bibliography{Bibliography}

\def\cprime{$'$}
\begin{thebibliography}{10}
\expandafter\ifx\csname url\endcsname\relax
  \def\url#1{\texttt{#1}}\fi
\expandafter\ifx\csname urlprefix\endcsname\relax\def\urlprefix{URL }\fi

\bibitem{bachiller2015family}
D.~Bachiller, F.~{Ced\'o}, E.~Jespers, J.~{Okni\'nski}, A family of
  irretractable square-free solutions of the {Yang-Baxter} equation, Forum
  Math. 29~(6) (2017) 1291--1306.
\newline\urlprefix\url{https://doi.org/10.1515/forum-2015-0240}

\bibitem{bax72}
R.~J. Baxter, Partition function of the eight-vertex lattice model, Annals of
  Physics 70~(1) (1972) 193--228.

\bibitem{cacsp2018}
M.~Castelli, F.~Catino, G.~Pinto, Indecomposable involutive set-theoretic
  solutions of the {Yang-Baxter} equation, J. Pure Appl. Algebra 220~(10)
  (2019) 4477--4493.
\newline\urlprefix\url{https://doi.org/10.1016/j.jpaa.2019.01.017}

\bibitem{castelli2021indecomposable}
M.~Castelli, F.~Catino, P.~Stefanelli, {Indecomposable Involutive Set-Theoretic
  Solutions of the Yang--Baxter Equation and Orthogonal Dynamical Extensions of
  Cycle Sets}, Mediterr. J. Math. 18~(6) (2021) 1--27.
\newline\urlprefix\url{https://doi.org/10.1007/s00009-021-01912-4}

\bibitem{capiru2020}
M.~Castelli, G.~Pinto, W.~Rump, On the indecomposable involutive set-theoretic
  solutions of the {Y}ang-{B}axter equation of prime-power size, Comm. Algebra
  48~(5) (2020) 1941--1955.
\newline\urlprefix\url{https://doi.org/10.1080/00927872.2019.1710163}

\bibitem{cedo2014braces}
F.~Ced{\'o}, E.~Jespers, J.~Okni{\'n}ski, Braces and the {Yang-Baxter}
  equation, Comm. Math. Phys. 327~(1) (2014) 101--116.
\newline\urlprefix\url{https://doi.org/10.1007/s00220-014-1935-y}

\bibitem{cedo2020primitive}
F.~Ced{\'o}, E.~Jespers, J.~Okninski, Primitive set-theoretic solutions of the
  {Yang-Baxter} equation, Comm. Math.
\newline\urlprefix\url{https://doi.org/10.1142/S0219199721501054}

\bibitem{cedo2021constructing}
F.~Ced{\'o}, J.~Okni{\'n}ski, Constructing finite simple solutions of the
  {Yang-Baxter} equation, Adv. Math. 391 (2021) 107968.
\newline\urlprefix\url{https://doi.org/10.1016/j.aim.2021.107968}

\bibitem{cedo2022new}
F.~Ced{\'o}, J.~Okni{\'n}ski, New simple solutions of the {Yang-Baxter}
  equation and solutions associated to simple left braces, J. Algebra 600
  (2022) 125--151.
\newline\urlprefix\url{https://doi.org/10.1016/j.jalgebra.2022.02.011}

\bibitem{drinfeld1992some}
V.~G. Drinfel\cprime~d, On some unsolved problems in quantum group theory, in:
  Quantum groups ({L}eningrad, 1990), vol. 1510 of Lecture Notes in Math.,
  Springer, Berlin, 1992, pp. 1--8.
\newline\urlprefix\url{https://doi.org/10.1007/BFb0101175}

\bibitem{etingof1998set}
P.~Etingof, T.~Schedler, A.~Soloviev, Set-theoretical solutions to the {Quantum
  Yang-Baxter} equation, Duke Math. J. 100~(2) (1999) 169--209.
\newline\urlprefix\url{http://doi.org/10.1215/S0012-7094-99-10007-X}

\bibitem{gateva1998semigroups}
T.~Gateva-Ivanova, M.~Van~den Bergh, Semigroups of {I-Type}, J. Algebra 206~(1)
  (1998) 97--112.
\newline\urlprefix\url{https://doi.org/10.1006/jabr.1997.7399}

\bibitem{JePiZa20x}
P.~Jedli{\v{c}}ka, A.~Pilitowska, A.~Zamojska-Dzienio, Indecomposable
  involutive solutions of the {Y}ang-{B}axter equation of multipermutational
  level $2$ with abelian permutation group, Forum Math. 2020.
\newline\urlprefix\url{https://doi.org/10.1515/forum-2021-0130}

\bibitem{jedlivcka2021cocyclic}
P.~Jedli{\v{c}}ka, A.~Pilitowska, A.~Zamojska-Dzienio, Cocyclic braces and
  indecomposable cocyclic solutions of the {Yang-Baxter} equation, arXiv
  preprint arXiv:2107.12319.
\newline\urlprefix\url{https://arxiv.org/pdf/2107.12319.pdf}

\bibitem{Okninski2022}
J.~{Okni\'nski}, {"Braces in Bracelet Bay" (LMS Regional Meeting and Workshop -
  Swansea University)}, Talk.
\newline\urlprefix\url{https://sites.google.com/view/lmsmeetingbracesinbraceletbay/home}

\bibitem{rump2005decomposition}
W.~Rump, A decomposition theorem for square-free unitary solutions of the
  quantum {Y}ang-{B}axter equation, Adv. Math. 193 (2005) 40--55.
\newline\urlprefix\url{https://doi.org/10.1016/j.aim.2004.03.019}

\bibitem{rump2006modules}
W.~Rump, Modules over braces, Algebra Discrete Math. -~(2) (2006) 127--137.
\newline\urlprefix\url{http://admjournal.luguniv.edu.ua/index.php/adm/article/view/892/421}

\bibitem{rump2007braces}
W.~Rump, Braces, radical rings, and the quantum {Y}ang-{B}axter equation, J.
  Algebra 307~(1) (2007) 153--170.
\newline\urlprefix\url{https://doi.org/10.1016/j.jalgebra.2006.03.040}

\bibitem{rump2020}
W.~Rump, Classification of indecomposable involutive set-theoretic solutions to
  the {Y}ang-{B}axter equation, Forum Math. 32~(4) (2020) 891--903.
\newline\urlprefix\url{https://doi.org/10.1515/forum-2019-0274}

\bibitem{Ru20}
W.~Rump, Cocyclic solutions to the {Y}ang-{B}axter equation, Proc. Amer. Math.
  Soc. 149~(2) (2021) 471--479.
\newline\urlprefix\url{https://doi.org/10.1090/proc/15220}

\bibitem{smock}
A.~Smoktunowicz, A.~Smoktunowicz, Set-theoretic solutions of the
  {Y}ang-{B}axter equation and new classes of {R}-matrices, Linear Algebra
  Appl. 546 (2018) 86--114.
\newline\urlprefix\url{https://doi.org/10.1016/j.laa.2018.02.001}

\bibitem{vendramin2016extensions}
L.~Vendramin, Extensions of set-theoretic solutions of the {Y}ang-{B}axter
  equation and a conjecture of {G}ateva-{I}vanova, J. Pure Appl. Algebra 220
  (2016) 2064--2076.
\newline\urlprefix\url{https://doi.org/10.1142/S1005386716000183}

\bibitem{Ve15pack}
L.~Vendramin, A.~Konovalov, {C}ombinatorial {S}olutions for the {Y}ang-{B}axter
  equation, {V}ersion 0.9.0 ({G}{A}{P} {p}ackage {Y}ang{B}axter) (2019).
\newline\urlprefix\url{https://gap-packages.github.io/YangBaxter}

\bibitem{yang1967}
C.~N. Yang, {Some Exact Results for the Many-Body Problem in one Dimension with
  Repulsive Delta-Function Interaction}, Phys. Rev. Lett. 19 (1967) 1312--1315.
\newline\urlprefix\url{https://link.aps.org/doi/10.1103/PhysRevLett.19.1312}

\end{thebibliography}

\end{document}